\documentclass[a4paper,11pt]{amsart}
\usepackage{amssymb,mathrsfs}
\usepackage{amsthm}
\usepackage{amsmath}
\usepackage{latexsym}
\usepackage{fancyhdr}
\usepackage{ascmac}
\usepackage{color}
\usepackage[all]{xy}
\usepackage{amscd}

\theoremstyle{plain}
\newtheorem{thm}{Theorem}[section]
\newtheorem{prop}[thm]{Proposition}
\newtheorem{cor}[thm]{Corollary}
\newtheorem{lem}[thm]{Lemma}

\newtheorem{rmk}[thm]{Remark}
\makeatletter

\@addtoreset{equation}{section}
\makeatother

\newcommand{\bQ}{\overline{\mathbb{Q}}}

\newcommand{\bF}{\overline{\mathbb{F}}}

\newcommand{\C}{\mathbb{C}}

\newcommand{\Q}{\mathbb{Q}}

\newcommand{\Z}{\mathbb{Z}}

\newcommand{\F}{\mathbb{F}}

\newcommand{\lra}{\longrightarrow}

\newcommand{\diag}{{\rm diag}}
\newcommand{\ds}{\displaystyle}

\newcommand{\G}{\Gamma}

\newcommand{\bS}{\overline{S}}

\newcommand{\GL}{{\rm GL}}
\newcommand{\SL}{{\rm SL}}

\newcommand{\bs}{\backslash}

\newcommand{\PSL}{{\rm PSL}}

\title[The $L$-function of the surface parametrizing cuboids]
{The $L$-function of the surface parametrizing cuboids}
\author{Madoka Horie and Takuya Yamauchi}
\thanks{}

\address{Madoka Horie \\
Faculty of Fundamental Science, 
National Institute of Technology, 
Niihama College., 102-8554, JAPAN}
\email{horiemaaa@gmail.com}

\address{Takuya Yamauchi \\ 
Mathematical Inst. Tohoku Univ.\\
 6-3,Aoba, Aramaki, Aoba-Ku, Sendai 980-8578, JAPAN}
\email{takuya.yamauchi.c3@tohoku.ac.jp}

\begin{document}

\maketitle

\begin{abstract}
In this note, we compute the $L$-function of the projective smooth surface $S$ over $\Q$
that parametrizes cuboids whose geometric properties are studied in detail by Stoll and Testa \cite{ST}.
As a byproduct, we completely determine the structure of
${\rm Pic}(S_{\overline{\Q}})$ as a ${\rm Gal}(\overline{\Q}/\Q)$-module.
\end{abstract}

\section{Introduction}\label{intro} 
The congruent number problem has fascinated number theorists for centuries. 
It asks which positive integers $n$ arise as the area of a right triangle with rational side lengths. 
This classical problem can be reduced to the existence of rational solutions of certain Diophantine equations and, ultimately, to the question of whether the Mordell--Weil rank of the elliptic curve $E_n:y^2=x^3-n^2 x$ is positive. The study of this problem involves modern mathematics, including the theory of elliptic curves and modular forms \cite{Kob}. 

A higher-dimensional analogue is the rational box problem, which asks for the existence
of rational solutions to the relations among the three side lengths $a_1,a_2,a_3$,
the three face diagonals $b_1,b_2,b_3$, and the space diagonal $c$ of a
three-dimensional rectangular box (see \cite[Section~0]{NS} for the history and also 
\cite{Kani}, \cite{PW},\cite{B},\cite{FM}).
Namely, these quantities satisfy
\begin{equation}\label{Box} 
\left\{
\begin{array}{l}
a^2_1+a^2_2=b_3^2\\
a^2_1+a^2_3=b_2^2\\
a^2_2+a^2_3=b_1^2\\
a^2_1+a^2_2+a^2_3=c^2.
\end{array}\right.
\end{equation}
A solution of (\ref{Box}) is called a cuboid, and it is called a rational cuboid if
$a_1,a_2,a_3,b_1,b_2,b_3,c$ are all rational. As in the case of congruence number problem, 
we can attach the projective algebraic surfaces parametrizing cuboids that is defined by 
\begin{equation}\label{bSeq1}
\bar{S}: 
\left\{
\begin{array}{l}
a^2_1+b^2_1=c^2\\
a^2_2+b^2_2=c^2\\
a^2_3+b^2_3=c^2\\
a^2_1+a^2_2+a^2_3=c^2
\end{array}\right.
\end{equation}
inside $\mathbb{P}^6$ with projective coordinates 
$[a_1:a_2:a_3:b_1:b_2:b_3:c]$. Note that (\ref{bSeq1}) is equivalent to (\ref{Box}) without 
any change of variables.  

In \cite{ST}, Stoll and Testa gave a detailed study of $\bar{S}$ and
computed many geometric invariants explicitly.
For example, they computed the Hodge diamonds of the minimal desingularization
$S$ of $\bar{S}$, as well as the Picard group of $S$ with explicit generators.
As a consequence, $S$ turns out to be of general type.
Therefore, if the Bombieri-Lang conjecture (\cite{Lang86},\cite[p.68, Conjecture 4.6]{DT}) is true,
there are only ``a few'' rational solutions on $S$, and hence on $\bar{S}$.
More precisely, $S(\Q)$ is not Zariski dense in $S_{\bar{\Q}}$. 
Although there are no known rational solutions of (\ref{bSeq1}) and only ``a few''
rational points are expected to exist by the Bombieri-Lang conjecture,
the geometry of the surface $\bar{S}$ remains of independent interest.

In this paper, we compute the $L$-functions of $S$ and $\bar{S}$.
As a byproduct, we give a complete description of the Picard group
${\rm Pic}_{\bar{\Q}}(S)$ as a $G_\Q := {\rm Gal}(\bQ/\Q)$-module.
We note that Stoll and Testa \cite{ST} determined explicit generators of
${\rm Pic}_{\bar{\Q}}(S)$ together with precise information on their fields of definition.
In principle, their results contain enough information to recover the Galois module structure. 
In fact, in the proof of \cite[Proposition 7]{ST}, by computer, they computed the intersection matrix for the set $\mathcal{G}$ (see \cite[p.7, Definition 6]{ST}) of possible generators and 
it may easy to extract precise generators from $\mathcal{G}$ and the shape of the intersection matrix. However, such a description is not given explicitly there.
Our approach provides a direct and explicit determination of the
$G_\Q$-module structure as a byproduct of the computation of the $L$-functions.

To state the claim, we prepare some notation. For a prime $\ell$, a non-negative integer $i$, and an algebraic 
variety $X$ over $\Q$, we denote by $H^i_{\text{\'et}}(X_{\bQ},\Q_\ell))$ the $i$-th $\ell$-adic 
\'etale cogomology of $V$ (cf. \cite{Deligne}).  
For an $\ell$-adic Galois representation $V$ of $G_\Q$, 
we denote by $L(s,V)$ the $L$-function of $V$ defined in \cite{Serre} where 
$s$ is the complex parameter. 
For a modular form $h$ (resp. the Dirichlet character $\chi_L$ associated to a quadratic extension 
$L/\Q$), we denote by 
$L(s,h)$ (resp. $L(s,\chi_L)$) the $L$-function of $h$ (resp. $\chi_L$).  
The Riemann zeta function is denoted by $\zeta(s)$. We refer to \cite[Section 4.4 and Section 5.9]{DS} for the $L$-functions.  For $N\in\{8,16,32\}$, let $h_N\in S_3(\G_1(N))$ be the unique newform of level $N$ with 
rational Fourier coefficients  
 (see Section \ref{etacoh}). 

We are now ready to claim the following result. 
\begin{thm}\label{main}{\rm(}Corollary \ref{PN}, \ref{LSS}{\rm)} Let $\ell$ be any prime. Then, it holds that 
\begin{enumerate}
\item $L(s,H^2_{\text{\'et}}(S_{\bQ},\Q_\ell))=
L(s,H^2_{\text{\'et}}(\bar{S}_{\bQ},\Q_\ell))\zeta(s-1)^{24}L(s-1,\chi_{\Q(\sqrt{-1})})^{24}$;
\item $L(s,H^2_{\text{\'et}}(\bar{S}_{\bQ},\Q_\ell))=L(s,h_{16})^3 L(s,h_{32})L(s,h_8)^3 L(s-1,L_\ell)$
where \\ 
$L(s,L_\ell)=\zeta(s)^{10}
L(s,\chi_{\Q(\sqrt{-1})})^{2}L(s,\chi_{\Q(\sqrt{-2})})L(s,\chi_{\Q(\sqrt{2})})^{3}$ 
{\rm(}see Theorem \ref{maineta} for $L_\ell${\rm )};
\item ${\rm Pic}(S_{\bQ})$ is free of rank $64$ and it is generated by 
\begin{enumerate}
\item 34 irreducible divisors defined over $\Q$;
\item 26 irreducible divisors strictly defined over $\Q(\sqrt{-1})$;
\item one irreducible divisor strictly defined over $\Q(\sqrt{-2})$;
\item 3 irreducible divisors strictly defined over $\Q(\sqrt{2})$.
\end{enumerate}
\end{enumerate}
\end{thm}

Arithmetic studies of Picard groups and N\'eron--Severi groups of algebraic surfaces
have been developed in parallel with advances in computational techniques
(cf.\ \cite{EJ11}, \cite{EJ15}, among others for K3 surfaces).
Beyond the general theory, it would be interesting to extend the results of
Kani--Schanz \cite{KS98} to diagonal quotient surfaces
$(X(N)\times X(N))/\Delta H$ for subgroups
$H \subset {\rm Aut}(X(N))$,
where $X(N)$ denotes the modular curve associated with the principal congruence
subgroup $\Gamma(N)$ of $\SL_2(\Z)$.
Our case corresponds to the special situation $N=8$ (see Section \ref{MF}).
We leave this problem for future work.

This paper is organized as follows.
In Section~\ref{SS}, we recall the results of Stoll and Testa \cite{ST} for $S$ and $\bar{S}$.
Using their results, we first compute the singular cohomology of $\bar{S}$ and thus the rank of 
the \'etale cohomology can be obtained by a comparison theorem.
In Section~\ref{MF}, we recall the modular covering of $\bar{S}$ due to Beauville
(see \cite[Section~4]{ST}) and we reconstruct the covering
in terms of elliptic modular cusp forms.
Finally, in Section~\ref{etacoh}, we compute the \'etale cohomology of $\bar{S}$ and $S$
in terms of elliptic cusp forms and the direct sum of certain twisted Tate motives. 
The main result is an immediate consequence of the results in this section. 

\subsection*{Acknowledgments}
The authors would like to thank Professor Testa for helpful comments.
They also thank Professor Riccardo Salvati Manni for informing them of his joint
results with Professor E.~Freitag concerning theta functions.

\section{The Surface $S$ parametrizing cuboids}\label{SS}
Recall the projective surface $\bS$ parametrizing cuboids defined by (\ref{bSeq1}). 
We view $\bar{S}$ as a projective geometrically connected algebraic variety over $\Q$. 
Let $Z$ be the singular locus of $\bar{S}$. It is defined over $\Q$ as a scheme. 
As studied in \cite{ST}, the surface $\bar{S}_{\bQ}$ has 48 isolated singularities. 
Exactly 24 points are defined over $\Q$ and others are strictly defined over $\Q(\sqrt{-1})$. 
Explicitly they are give by 
$${\small 
\begin{array}{llll}
\text{\small $[0:0:-1:-1:-1:0:1]$} & \text{\small $[0:0:-1:-1:1:0:1]$} & \text{\small $[0:0:-1:1:-1:0:1]$} & \text{\small $[0:0:-1:1:1:0:1]$} \\
\text{\small $[0:0:1:-1:-1:0:1]$} & \text{\small $[0:0:1:-1:1:0:1]$} & \text{\small $[0:0:1:1:-1:0:1]$} & \text{\small $[0:0:1:1:1:0:1]$} \\
\text{\small $[0:-1:0:-1:0:-1:1]$} & \text{\small $[0:-1:0:-1:0:1:1]$} & \text{\small $[0:-1:0:1:0:-1:1]$} & \text{\small $[0:-1:0:1:0:1:1]$} \\
\text{\small $[0:1:0:-1:0:-1:1]$} & \text{\small $[0:1:0:-1:0:1:1]$} & \text{\small $[0:1:0:1:0:-1:1]$} & \text{\small $[0:1:0:1:0:1:1]$} \\
\text{\small $[-1:0:0:0:-1:-1:1]$} & \text{\small $[-1:0:0:0:-1:1:1]$} & \text{\small $[-1:0:0:0:1:-1:1]$} & \text{\small $[-1:0:0:0:1:1:1]$} \\
\text{\small $[1:0:0:0:-1:-1:1]$} & \text{\small $[1:0:0:0:-1:1:1]$} & \text{\small $[1:0:0:0:1:-1:1]$} & \text{\small $[1:0:0:0:1:1:1]$} \\
\text{\small $[0:-1:-i:0:-i:1:0]$} & \text{\small $[0:-1:-i:0:i:1:0]$} & \text{\small $[0:-1:i:0:-i:1:0]$} & \text{\small $[0:-1:i:0:i:1:0]$} \\
\text{\small $[0:1:-i:0:-i:1:0]$} & \text{\small $[0:1:-i:0:i:1:0]$} & \text{\small $[0:1:i:0:-i:1:0]$} & \text{\small $[0:1:i:0:i:1:0]$} \\
\text{\small $[-1:0:-i:-i:0:1:0]$} & \text{\small $[-1:0:-i:i:0:1:0]$} & \text{\small $[-1:0:i:-i:0:1:0]$} & \text{\small $[-1:0:i:i:0:1:0]$} \\
\text{\small $[1:0:-i:-i:0:1:0]$} & \text{\small $[1:0:-i:i:0:1:0]$} & \text{\small $[1:0:i:-i:0:1:0]$} & \text{\small $[1:0:i:i:0:1:0]$} \\
\text{\small $[-1:-i:0:-i:1:0:0]$} & \text{\small $[-1:-i:0:i:1:0:0]$} & \text{\small $[-1:i:0:-i:1:0:0]$} & \text{\small $[-1:i:0:i:1:0:0]$} \\
\text{\small $[1:-i:0:-i:1:0:0]$} & \text{\small $[1:-i:0:i:1:0:0]$} & \text{\small $[1:i:0:-i:1:0:0]$} & \text{\small $[1:i:0:i:1:0:0]$}
\end{array}
}
$$
where we put $i:=\sqrt{-1}$. 

We denote by $S$ the minimal resolution of $\bar{S}$ along $Z$ which is also defined over $\Q$ and 
let $\pi:S\lra \bar{S}$ be the corresponding birational, proper surjective map. 
For an algebraic variety $X$,  we simply write $H^i(X)$ for the $i$-th singular cohomology 
$H^i(X(\C),\Q)$. 

We first compute the rank of $H^2(\bar{S})$ which immediately follows from the Hodge diamond 
of $S$ given in \cite[Section 2]{ST}. 
\begin{lem}\label{bS} The rank of $H^2(\bar{S})$ is 30 and the cohomology 
$H^2(\bar{S})$ is pure of weight 2 in the sense of \cite[Definition 5.40, p.131]{PS}.
\end{lem}
\begin{proof}We apply the Leray spectral sequence $E^{p,q}_2=H^p(\bar{S},R^q \pi_\ast \Q )
\Longrightarrow H^{p+q}(S,\Q)$. 
Since 
$$R^0\pi_\ast \Q=\Q,\ R^1\pi_\ast \Q=0,\ 
R^2\pi_\ast \Q=\bigoplus_{p\in Z(\C)}\Q(-1)$$
we have 
$E^{2,0}_2=H^2(S,\Q)$, $E^{1,1}_2=0$, and $E^{0,2}_2=\ds\bigoplus_{p\in Z(\C)}\Q(-1)$. 
There exists a common open subvariety $U$ of $S$ and $\bar{S}$ under $\pi$ 
such that $\bar{S}=U\coprod Z$ and $S=U\coprod Z'$ where $Z'$ is the proper transform of 
$Z$ under $\pi$. Then, by the excision theorem , we have an exact sequence 
$$H^2(Z,\Q)=0\lra H^3_c(U,\Q)\lra H^3(\bar{S},\Q)\lra H^3(Z,\Q)=0$$
so that the middle map is an isomorphism.
By the Poincar\'e duality (cf. \cite[Chapter 3, Theorem 3.35]{Ha}), we have $H^3_c(U,\Q)\simeq 
H_1(U,\Q)$. Since $S$ is connected by \cite[the proof of Proposition 7]{ST} 
and $Z'$ is a union of projective lines, $H_1(U,\Q)=0$. Thus 
$E^{3,0}_2=H^3(\bar{S},\Q)=0$.

By \cite[p.328, Theorem 5.12]{CE}), 
we have an exact sequence 
$$0\lra E^{0,2}_2\lra H^2(S)\lra E^{2,0}_2=H^2(\bar{S})\lra E^{3,0}_2=0.$$
Thus, the claim follows from the above exact sequence with 
${\rm rank}(H^2(S))=78$ and ${\rm rank}(E^{0,2}_2)=48$. 
\end{proof}

\section{Modular forms related to $S$}\label{MF}
We refer to \cite{DS} for the basic facts of elliptic modular forms and modular curves. 
For each positive integer $N$, we denote by $\G(N)\supset \G_0(N)\supset \G_1(N)$ the three kinds of congruence subgroups inside $\SL_2(\Z)$ introduced in \cite[p.13, Definition 1.2.1]{DS}. For such a congruence subgroup $\G$, we define the 
(compact) modular curve $X(\G):=\G\bs (\mathbb{H}\cup\mathbb{P}^1(\Q))$ 
and put 
$X(N)=X(\G(N))$ 
 for simplicity.  
We denote by $S_2(\G)$ the space of all cusp forms of weight 2 with respect to $\G$. 
It is known that we have an isomorphism $S_2(\G)\stackrel{\sim}{\lra}H^0(X(\G),\Omega^1_{X(\G)}),\ f(\tau)\mapsto f(\tau)d\tau$ where the right hand side stands for the space of 
all holomorphic 1-forms on $X(\G)$. When $\G=\G_1(N)$, we have 
the decomposition $S_2(\G_1(N))\simeq \ds\bigoplus_{\chi\in \widehat{(\Z/N\Z)}}S_2(\G_0(N),\chi)$ 
(see \cite[Section 5.2, p.169]{DS} ). 

As explained in \cite[Section 4]{ST}, the surface $\bar{S}$ is isomorphic over $\Q(\sqrt{-1})$ to 
the quotient variety $X(8)\times X(8)/\Delta G$ where $G$ is the kernel of the natural surjection 
$\PSL_2(\Z/8\Z)\lra \PSL_2(\Z/4\Z)$ and $\Delta G$ denotes the image of 
the diagonal embedding from $G$ to $G\times G$. 

In this section, we find a model $X$ of $X(8)$ defined over $\Q$ as an algebraic curve. Note that the genus of $X(8)$ is five so that 
${\rm dim}(S_2(\G(8)))=5$ (cf. \cite[Section 3.9, p.106]{DS}).  
Let $d_8:=\diag(8,1)\in M_2(\Z)$ and we see that  
$$\Gamma_1(64)\subset \G(8)':=d^{-1}_8\G(8)d_8=\G_0(64)\cap \Gamma_1(8)\subset \Gamma_0(64).$$

For each quadratic extension $L/\Q$, 
we denote by $\chi_L:G_\Q\lra \{\pm\}$ the quadratic character associated to $L/\Q$. 
Let $\chi:(\Z/64\Z)^\times\stackrel{{\rm mod}\ 8}{\lra} (\Z/8\Z)^\times =\langle 5,-1 \rangle \lra \C^\times$ be 
the quadratic character of conductor 8 defined by $\chi(5)=-1$ and $\chi(-1)=1$. 
The character $\chi_{\Q(\sqrt{2})}$ factors through
$$G_\Q \stackrel{{\rm res}}{\lra} {\rm Gal}(\Q(\zeta_{64})/\Q)
 \stackrel{{\rm res}}{\lra} {\rm Gal}(\Q(\zeta_{8})/\Q) \stackrel{{\rm res}}{\lra} {\rm Gal}(\Q(\sqrt{2})/\Q)\simeq\{\pm 1\}.$$
 Further, the restriction map ${\rm Gal}(\Q(\zeta_{64})/\Q)
 \stackrel{{\rm res}}{\lra} {\rm Gal}(\Q(\zeta_{8})/\Q)$ can be identified with 
 the natural projection $(\Z/64\Z)^\times\stackrel{{\rm mod} 8}{\lra} (\Z/8\Z)^\times$. 
Thus, we can regard $\chi$ with $\chi_{\Q(\sqrt{2})}$. 

Notice $d_8$ induces the isomorphism $S_2(\G(8))\stackrel{\sim}{\lra} S_2(\G(8)')$ by $f(\tau)\mapsto f(8\tau)$. 
Then, by newform theory and dimension comparison, we have  
\begin{equation}\label{decomp}
S_2(\G(8))\stackrel{\sim}{\lra} S_2(\G(8)')
\simeq \langle f_{32},f^{(2)}_{32} \rangle\oplus S_2(\G_0(64))\oplus S_2(\G_0(64),\chi)
\end{equation}
where $f_{32}$ is the newform in $S_2(\G_0(32))$ and $f^{(2)}_{32}$ is defined by 
$f^{(2)}_{32}(\tau)=f_{32}(2\tau)$. 
Let $f_{64}$ be the  newform in  $S_2(\G_0(64))$ and 
$g_{64,\pm 2\sqrt{-1}}$ be two newforms in $S_2(\G_0(64),\chi)$.  These five forms make up   
a basis of $S_2(\G(8)')$. By using the database \cite{LMFDB}, each form is explicitly given by 
\begin{equation}\label{qexp}
\begin{array}{l}
f_{32}(\tau)=q - 2  q^5 - 3  q^9 + 6  q^{13} + 2  q^{17} - q^{25}\cdots  \\
f^{(2)}_{32}(\tau)=q^2 - 2 q^{10} - 3  q^{18} + 6  q^{26} + 2  q^{34} - q^{50}\cdots \\ 
f_{64}(\tau)=q + 2  q^5 - 3  q^9 - 6  q^{13} + 2  q^{17} - q^{25}\cdots \\ 
g_{64,\pm}(\tau)=q + b q^3 - q^9 - 3b q^{11} - 6  q^{17} + b  q^{19} + 5  q^{25}\cdots,\ b=\pm 2\sqrt{-1}.
\end{array}
\end{equation}
It is easy to see that 
\begin{equation}\label{g64twist}
f_{64}=f_{32}\otimes\chi_{\Q(\sqrt{\pm 2})},\ g_{64,-}=g_{64,+}\otimes \chi,\ \chi=\chi_{\Q(\sqrt{2})}
\end{equation}
The newforms $f_{32}$ and $f_{64}$ have complex multiplication by $\Q(\sqrt{-1})$ 
whereas the newforms $g_{64,\pm}$ have  complex multiplication by $\Q(\sqrt{-2})$.  
It is easy to see that 
\begin{equation}\label{twist-inv}
f_{32}\otimes \chi_{\Q(\sqrt{-1})}=f_{32},\ 
f_{64}\otimes \chi_{\Q(\sqrt{-1})}=f_{64},\ 
g_{64,+}\otimes \chi_{\Q(\sqrt{-1})}=g_{64,-}.
\end{equation}

Since $X(8)$ is non-hyperelliptic  (cf \cite[Theorem 4.1]{BKX}), 
its canonical divisor is very ample.
Hence, through the isomorphism $S_2(\G(8)')\stackrel{\sim}{\lra}H^0(X(\G(8)'),
\Omega^1_{X(\G(8)')}),\ f(\tau)\mapsto f(\tau)d\tau$, we have a canonical embedding 
$\Phi:X(\G(8)')\hookrightarrow \mathbb{P}^4(\C)$ given by $\tau\mapsto 
[x:y:u:v:w]$ where 
$$(x,y,u,v,w):=\Big(\frac{1}{2}\Big(g_{64,+}(\tau)+g_{64,-}(\tau)\Big), 
\frac{1}{\sqrt{-1}}\Big(g_{64,+}(\tau)-g_{64,-}(\tau)\Big), f_{32}(\tau),f^{(2)}_{32}(\tau),
f_{64}(\tau)\Big).$$
By using Petri's theorem, the image of $\Phi$ is given by the intersection of 
three quadratic equations in $\mathbb{P}^4$. By determining 
the coefficients of those equations with (\ref{qexp}) (see \cite{SM} for computational details), 
we can recover the model $X$ of $X(8)$ given in \cite[Section 4]{ST}: 
\begin{equation}\label{bSeq}
X: 
\left\{
\begin{array}{l}
u^2=2xy\\
v^2=x^2-y^2\\
w^2=x^2+y^2
\end{array}\right.
\end{equation}
inside $\mathbb{P}^4$. 
As explained in \cite[Section 4, p.~9]{ST}, there is an explicit isomorphism
\[
X(8)\times X(8)/\Delta G \;\simeq\; \bar{S}
\]
defined over $\Q(\sqrt{-1})$.
On the other hand, by twisting $X(8)\times X(8)$ by an involution defined over $\Q$,
they also showed that the quotient of the Weil restriction
\[
{\rm Res}_{\Q(\sqrt{-1})/\Q}(X(8)_{\Q(\sqrt{-1})})
\]
by a certain finite group 
is isomorphic to $\bar{S}$ over $\Q$.
However, the \'etale cohomology of the Weil restriction ${\rm Res}_{\Q(\sqrt{-1})/\Q}(X(8)_{\Q(\sqrt{-1})})$ 
is not so easy to compute directly. We detour this situation by using modular forms and its 
Galois representations. 
\section{The \'etale cohomology of $S$}\label{etacoh}
We fix a rational prime $\ell$ and for an algebraic variety $Y$ over $\Q$, we simply write $H^i_{\text{\'et}}(Y)$  for 
the $i$-th $\ell$-adic \'etale cohomology $H^i_{\text{\'et}}(Y_{\bQ},\Q_\ell)$. 
We also write $H^i_{\text{\'et}}(Y)|_{G_L}$ when we view it as a $\Q_\ell[G_L]$-module for 
a finite extension $L/\Q$ inside $\bQ$. Here $G_L={\rm Gal}(\bQ/L)$. 
We also fix embeddings $K:=\Q(\sqrt{-1}) \hookrightarrow \overline{\Q},\  
\overline{\Q} \hookrightarrow \C$, and $\bQ \hookrightarrow \bQ_\ell$ that are compatible with a fixed isomorphism $\overline{\Q}_\ell \simeq \C$.

Let $s:{\rm Res}_{K/\Q}(X(8)_{K})\lra \bar{S}$ be the quotient map defined over $\Q$ which is finite and surjective. 
It induces an injection 
$$s^\ast:H^2_{\text{\'et}}(\bar{S})\lra H^2_{\text{\'et}}({\rm Res}_{\Q(\sqrt{-1})/\Q}(X(8)_{\Q(\sqrt{-1})}))$$
as a $\Q_\ell[G_\Q]$-module. 
We view this homomorphism as a $\Q_\ell[G_K]$-module and write it as $s^\ast$ again. 
Then, we have an injection 
$$s^\ast:H^2_{\text{\'et}}(\bar{S})|_{G_K}\lra 
H^2_{\text{\'et}}({\rm Res}_{K/\Q}(X(8)_{K}))|_{G_K}=
H^2_{\text{\'et}}({\rm Res}_{K/\Q}(X(8)_{K})_K)|_{G_K}=
H^2_{\text{\'et}}(X(8)_K\times X(8)_K)|_{G_K}$$
as a $\Q_\ell[G_K]$-module. 
By K\"unneth decomposition, we have 
\begin{equation}\label{Kdecomp}
H^2_{\text{\'et}}(X(8)_K\times X(8)_K)\simeq \Q_\ell(-1)\oplus   \Big(
H^1_{\text{\'et}}(X(8)_K)\otimes H^1_{\text{\'et}}(X(8)_K)\Big)
 \oplus\Q_\ell(-1)
 \end{equation}
 as a $\Q_\ell[G_K]$-module. 
Here, $\Q_\ell(-1)$ is a one-dimensional $\Q_\ell$-vector space, so that the geometric Frobenius ${\rm Frob}_p$ at $p \ne \ell$ acts on it as multiplication by $p$. 
Taking the induced representations to $G_\Q$, we see that  
${\rm Ind}^{G_\Q}_{G_K}(H^2_{\text{\'et}}(\bar{S})|_{G_K})=
H^2_{\text{\'et}}(\bar{S})\oplus H^2_{\text{\'et}}(\bar{S})(\chi_K)$ 
isomorphic to a certain submodule of  
\begin{equation}\label{Kdecomp2}
 \Q_\ell(-1)^{\oplus 2}\oplus  \Q_\ell(\chi_K)(-1)^{\oplus 2}
\oplus   
(H^1_{\text{\'et}}(X(8))\otimes H^1_{\text{\'et}}(X(8)))\oplus 
(H^1_{\text{\'et}}(X(8))\otimes H^1_{\text{\'et}}(X(8))(\chi_K))
 \end{equation} 
as a $\Q_\ell[G_\Q]$-module. 
By \cite[Section 1, p.349]{Fal}, $H^1_{\text{\'et}}(X(8))$ is a semisimple $\Q_\ell[G_\Q]$-module 
and so is (\ref{Kdecomp2}). 
Thus, $H^2_{\text{\'et}}(\bar{S})$ is also  a semisimple $\Q_\ell[G_\Q]$-module. Further, the composition map of $s$ and 
the birational map from $\bar{S}$ to $S$ yields an injection from $H^2_{\text{\'et}}(S)$ to (\ref{Kdecomp2}) as a 
$\Q_\ell[G_\Q]$-module (it follows easily from the excision). 
Thus,  $H^2_{\text{\'et}}(S)$ is also a semisimple $\Q_\ell[G_\Q]$-module. Summing up, we have the following result:
\begin{prop}\label{semisimple}Keep the notation being as above. Then, 
 $H^2_{\text{\'et}}(S)$ and  $H^2_{\text{\'et}}(\bar{S})$ are semisimple $\Q_\ell[G_\Q]$-modules. 
\end{prop} 
 
For each newform $f=\ds\sum_{n\ge 0}a_n(f)q^n\in S_k(\G_0(N),\chi)$ with $k\ge 2$, we 
can associate a unique $\ell$-adic Galois representation 
$$\rho_{f,\ell}:G_\Q\lra \GL_2(\bQ_\ell)$$ such that 
$${\rm tr}(\rho_{f,\ell}({\rm Frob}_p))=a_p(f),\ 
\det(\rho_{f,\ell}({\rm Frob}_p))=\chi(p) p^{k-1}$$
for each prime $p\nmid \ell N $ (cf. \cite{Ribet77}). 
Let $V_{f,\ell}$ be the representation space of 
$\rho_{f,\ell}$.  
It is well-known that $V_{f,\ell}$ is a simple $\bQ_\ell[G_\Q]$ module (\cite[Theorem (2.3)]{Ribet77}). 
By using (\ref{decomp}) and \cite[p.46, Proposition (2.3)]{Ribet80}, we see that 
$$H^1_{\text{\'et}}(X(8))\otimes_{\Q_\ell}\bQ_\ell\simeq V^{\oplus 2}_{f_{32},\ell}\oplus V_{f_{64},\ell}\oplus 
V_{g_{64,+},\ell}\oplus 
V_{g_{64,-},\ell}.$$
Further, by (\ref{twist-inv}), 
$$H^1_{\text{\'et}}(X(8))\otimes_{\Q_\ell}\bQ_\ell\simeq 
(H^1_{\text{\'et}}(X(8))(\chi_K))\otimes_{\Q_\ell}\bQ_\ell$$
as a $\Q_\ell[G_\Q]$-module. 
Thus, $H^2_{\text{\'et}}(\bar{S})\otimes_{\Q_\ell}\bQ_\ell$ and 
$H^2_{\text{\'et}}(S)\otimes_{\Q_\ell}\bQ_\ell$ are a semisimple submodule of 
 $H^2_{\text{\'et}}(X(8) \times X(8))\otimes_{\Q_\ell}\bQ_\ell$ as a $\bQ_\ell[G_\Q]$-module. 
\begin{prop}\label{tensor}Keep the notation being as above. As a $\bQ_\ell[G_\Q]$-module, it holds that 
\begin{enumerate}
\item For $f=f_{32}$ or $f_{64}$, 
$$V_{f,\ell}^{\otimes 2}\simeq V_{h_{16},\ell}\oplus\bQ_\ell(-1)^{\oplus2}$$ where 
$V_{h_{16},\ell}$ denotes a unique $\ell$-adic representation attached the newform 
$h_{16}$ with CM by 
$\Q(\sqrt{-1})$  in $S_3(\G_0(16),\chi_{4})$ where $\chi_{4}$ is the quadratic character of 
$(\Z/16\Z)^\times$ with conductor 4 defined by  $\chi_4(5)=1$ and $\chi_4(-1)=-1$;
\item  For $g=g_{64,\pm }$ {\rm(}recall that the character of $g$ is $\chi=\chi_{\Q(\sqrt{2})}${\rm)}, 
$$V_{g, \ell}^{\otimes 2}\simeq V_{h_{32},\ell}\oplus (\chi\otimes(\bQ_\ell(-1)^{\oplus2}))$$ where 
$V_{h_{32},\ell}$ denotes a unique $\ell$-adic representation attached the newform 
$h_{32}$ with CM by 
$\Q(\sqrt{-2})$ in $S_3(\G_0(32),\chi_{8})$ where $\chi_{8}$ is the quadratic character of 
$(\Z/32\Z)^\times$ with conductor 8 defined by  $\chi_8(5)=-1$ and $\chi_8(-1)=-1$;
\item $V_{g_{64,+}, \ell}\otimes V_{g_{64,-}, \ell}\simeq 
(\chi\otimes V_{h_{32},\ell})\oplus\bQ_\ell(-1)^{\oplus2}$;
\item for $f=f_{32}$ or $f_{64}$ and $g=g_{64,\pm }$, 
the tensor product 
 $V_{f,\ell}\otimes V_{g,\ell}$ is an irreducible $\bQ_\ell[G_\Q]$-module. 
 The same is true for $V_{f_{32},\ell}\otimes V_{f_{64},\ell}$. 
\end{enumerate}
\end{prop}
\begin{proof}For the first claim, 
since $f$ has CM by $K:=\Q(\sqrt{-1})$, it gives rise to a Hecke character $\psi:
G_{K}:={\rm Gal}(\bQ/K)\lra \bQ^\times_\ell$. 
Then, the newform $h_{16}$ corresponds to $\psi^2$ and it follows from this that 
$h_{16}$ belongs to $S_3(\G_1(64))$. 
Since both sides of the claim are semisimple $\bQ_\ell[G_\Q]$-modules 
(\cite[Theorem (2.3)]{Ribet77} again), we can check the claim by comparing 
Fourier coefficients of newforms in $S_3(\G_1(64))$ which have CM by $\Q(\sqrt{-1})$. 
The same strategy works for the second claim. 
The third claim follows from the second claim with (\ref{g64twist}). 

Since CM types are different each other, $V_{f,\ell}$ is not isomorphic to $V_{g,\ell}$ as 
a $\bQ_\ell[G_\Q]$-module. 
Thus, the cuspidal representations $\pi_f$ and $\pi_g$ are not equivalent each other. 
The fourth claim now follows from \cite[Section 3, Cuspidality Criterion]{Ra}.
\end{proof}
\begin{rmk}Contrary to Proposition \ref{tensor}-(3), $V_{f_{32},\ell}\otimes V_{f_{64},\ell}$ 
is irreducible even if $f_{64}=f_{32}\otimes \chi_{\Q(\sqrt{\pm 2})}$. The difference is 
whether or not the twist comes from the central character.  
\end{rmk}

We are now ready to determine $H^i_{\text{\'et}}(S)$ and $H^i_{\text{\'et}}(\bar{S})$ 
as a $\Q_\ell[G_\Q]$-module or as a $\bQ_\ell[G_\Q]$-module . 
\begin{thm}\label{maineta}
It holds that 
$$H^2_{\text{\'et}}(S)\simeq H^2_{\text{\'et}}(\bar{S})\oplus \Q_\ell(-1)^{\oplus 24}
\oplus  \Q_\ell(\chi_{\Q(\sqrt{-1})}(-1))^{\oplus 24}$$
and 
$$H^2_{\text{\'et}}(\bar{S})\otimes_{\Q_\ell}\bQ_\ell\simeq V_{h_{16},\ell}^{\oplus 3}\oplus 
 V_{h_{32},\ell}\oplus (\chi\otimes 
 V_{h_{32},\ell})^{\oplus 3}\oplus (L_\ell(-1)\otimes_{\Q_\ell}\bQ_\ell)$$
 where 
 $$L_\ell:=\Q_\ell^{\oplus 10}\oplus 
 \Q_\ell(\chi_{\Q(\sqrt{-1})})^{\oplus 2}
  \oplus 
 \Q_\ell(\chi_{\Q(\sqrt{-2})})
  \oplus 
\Q_\ell(\chi_{\Q(\sqrt{2})})^{\oplus 3}.$$
  For other degree, $H^i_{\text{\'et}}(S)=H^i_{\text{\'et}}(\bar{S})=\Q_\ell(-i/2)$ for $i=0,2$ 
  and $H^1_{\text{\'et}}(S)=H^1_{\text{\'et}}(\bar{S})=0$. 
  
We remark that $\chi\otimes 
 V_{h_{32},\ell}\simeq V_{h_8,\ell}$ where $h_8$ is the newform in  
  $S_3(\G_0(8),\chi_{8})$ and $\chi_{8}$ is the quadratic character of 
$(\Z/8\Z)^\times$ with conductor 8 defined by  $\chi_8(5)=-1$ and $\chi_8(-1)=-1$.
\end{thm}
\begin{proof}Since the first cohomology is invariant under the blowing up, 
we have $H^1_{\text{\'et}}(S)=H^1_{\text{\'et}}(\bar{S})=0$ by the Hodge diamond of $S$ and 
the comparison theorem. The claim for $H^i_{\text{\'et}}$ with $i=0,2$ 
is obvious.  

Let $p$ be any prime such that $p\nmid 2\ell$. 
Note that the inertia group at $p$ acts trivially on $H^2_{\text{\'et}}(X(8)\times X(8))$ 
and so is $H^2_{\text{\'et}}(\bar{S})$. By Leray spectral sequence for $\pi$, we see that $H^2_{\text{\'et}}(\bar{S})$ is a 
quotient of $H^2_{\text{\'et}}(S)$ by the direct sum of twisted 
Tate modules.  
Thus, we have 
$H^2_{\text{\'et}}(\bar{S})\simeq H^2_{\text{\'et}}(\bar{S}_{\bF_p})$ by the proper smooth 
base change for $H^2_{\text{\'et}}(S)$. 
By Lefschetz trace formula (\cite[Th\'eor\`eme 3.1.]{Deligne}), we have 
\begin{eqnarray}\label{Lef}
\sharp\bar{S}(\F_p)&=&p^2+1+{\rm tr}({\rm Frob}_p|H^2_{\text{\'et}}(\bar{S}_{\bF_p})) 
\nonumber\\
&=&p^2+1+{\rm tr}({\rm Frob}_p|H^2_{\text{\'et}}(\bar{S}))=
p^2+1+{\rm tr}({\rm Frob}_p|H^2_{\text{\'et}}(\bar{S})_{\Q_\ell}\bQ_\ell).
\end{eqnarray}
By \cite[Theorem 2]{ST} and the comparison theorem, 
the image of ${\rm Pic}(S_{\bQ})_\Q$ under the $\ell$-adic cycle map 
exhausts the transcendental part of $H^2_{\text{\'et}}(S)$. 
Thus, by Proposition \ref{tensor} (in particular, the fourth case can not occur), 
$H^2_{\text{\'et}}(\bar{S})\otimes_{\Q_\ell}\bQ_\ell$ can be written as 
$$V_{h_{16},\ell}^{\oplus a_1}\oplus 
 V_{h_{32},\ell}^{a_2}\oplus (\chi\otimes 
 V_{h_{32},\ell})^{a_3}\oplus (L_\ell(-1)\otimes_{\Q_\ell}\bQ_\ell)$$
 where 
$$L_\ell:=\Q_\ell^{\oplus a_4}\oplus \Q_\ell(\chi_{\Q(\sqrt{-1})})^{\oplus a_5}
  \oplus \Q_\ell(\chi_{\Q(\sqrt{-2})})^{\oplus a_6}
  \oplus \Q_\ell(\chi_{\Q(\sqrt{2})})^{\oplus a_7}$$ 
  for some non-negative integers $a_1,\ldots,a_7$. 
By computing the both sides of (\ref{Lef}) for $3\le p\le 97$ (we may choose $\ell$ to be greater than $97$), 
we can determine $a_1,\ldots,a_7$ and thus we have the claim. 
Note that the computation is done by Magma. 
\end{proof}
As an immediate consequence of Theorem \ref{maineta} with \cite[Theorem 2]{ST}, we have the followings result.
\begin{cor}\label{PN}The Picard group 
${\rm Pic}(S_{\bQ})$ is free of rank $64$ and it is generated by 
\begin{enumerate}
\item 34 irreducible divisors defined over $\Q$;
\item 26 irreducible divisors strictly defined over $\Q(\sqrt{-1})$;
\item one irreducible divisor strictly defined over $\Q(\sqrt{-2})$;
\item 3 irreducible divisors strictly defined over $\Q(\sqrt{2})$.
\end{enumerate}
\end{cor}
Finally, we state the result for the $L$-functions of $\bar{S}$ and $S$ which also follows from 
Theorem \ref{maineta}.
\begin{cor}\label{LSS}It holds that 
$$L(s,H^2_{\text{\'et}}(S))=
L(s,H^2_{\text{\'et}}(\bar{S}))\zeta(s-1)^{24}L(s-1,\chi_{\Q(\sqrt{-1})})^{24}$$
and 
$$
L(s,H^2_{\text{\'et}}(\bar{S}))=L(s,h_{16})^3 L(s,h_{32})L(s,h_8)^3L(s-1,L_\ell)$$
where 
$L(s,L_\ell)=\zeta(s)^{10}
L(s,\chi_{\Q(\sqrt{-1})})^{2}L(s,\chi_{\Q(\sqrt{-2})})L(s,\chi_{\Q(\sqrt{2})})^{3}$. 
\end{cor}

\end{document}